\documentclass[secnum,leqno]{rims-bessatsu}
\usepackage{amssymb, amsmath}
\usepackage{graphics}
\usepackage[dvips]{graphicx}
\usepackage{eepic}
\usepackage{epic}
\newtheorem{thm}{Theorem}[section]

\newtheorem{lmm}[thm]{Lemma}

\newtheorem{prp}[thm]{Proposition}

\theoremstyle{definition}

\newtheorem{exa}[thm]{Example}
\theoremstyle{remark}
\newtheorem*{rem}{Remark}
\def\stf#1#2{\left[#1\atop#2\right]} 
\def\sts#1#2{\left\{#1\atop#2\right\}}
\title{Bernoulli-Carlitz and Cauchy-Carlitz numbers with Stirling-Carlitz numbers}
\author{Hajime \textsc{Kaneko}\footnote{Institute of Mathematics, University of Tsukuba, 1-1-1 Tennodai, Tsukuba, Ibaraki 350-8571 Japan.\newline 
Research Core for Mathematical Sciences, University of Tsukuba, 1-1-1, Tennodai, Tsukuba, Ibaraki, 305-8571, Japan. 
\newline e-mail: \texttt{kanekoha@math.tsukuba.ac.jp}}
          ~and Takao \textsc{Komatsu}\footnote{ School of Mathematics and Statistics, Wuhan University, Wuhan 430072 China. \endgraf e-mail: \texttt{komatsu@whu.edu.cn}}}
\AuthorHead{Hajime Kaneko and Takao Komatsu}         
\support{The first author is supported by JSPS KAKENHI Grant Number 15K17505.}  
\VolumeNo{72}           
\YearNo{2018}           
\PagesNo{23--35}      
\communication{Received December 21, 2015.
Revised April 22, 2016.}      
\begin{document}
%

\maketitle

\begin{abstract}      
Recently, the Cauchy-Carlitz number was defined as the counterpart of the Bernoulli-Carlitz number. Both numbers can be expressed explicitly in terms of so-called Stirling-Carlitz numbers. In this paper,  
we study the second analogue of Stirling-Carlitz numbers and give some general formulae, including Bernoulli and Cauchy numbers in formal power series with complex coefficients, and Bernoulli-Carlitz and Cauchy-Carlitz numbers in function fields. We also give some applications of Hasse-Teichm\"uller derivative to hypergeometric Bernoulli and Cauchy numbers in terms of associated Stirling numbers.
\end{abstract}

\section{The second analogue of Stirling-Carlitz numbers}   

The (unsigned) Stirling numbers of the first kind $\stf{n}{k}$ and the Stirling numbers of the second kind $\sts{n}{k}$ are defined by the generating functions 
$$
\frac{\bigl(-\log(1-x)\bigr)^k}{k!}=\sum_{n=0}^\infty\stf{n}{k}\frac{x^n}{n!}\quad\hbox{and}\quad 
\frac{(e^x-1)^k}{k!}=\sum_{n=0}^\infty\sts{n}{k}\frac{x^n}{n!}\,,
$$ 
respectively.  Based upon these generating functions, in \cite{KK} we introduced Stirling-Carlitz numbers $\stf{n}{k}_C$ and $\sts{n}{k}_C$.  Using these $C$-Stirling-Carlitz numbers, Bernoulli-Carlitz numbers $BC_n$ (\ref{bcarlitz}) and Cauchy-Carlitz numbers $CC_n$ (\ref{ccarlitz}) can be expressed explicitly.  
Notice that Bernoulli-Carlitz numbers $BC_n$ (\cite{Car35,Car37,Car40,Goss}) are given by 
\begin{equation} 
\frac{z}{e_C(z)}=\sum_{n=0}^\infty\frac{BC_n}{\Pi(n)}z^n
\label{bcarlitz}
\end{equation}  
and Cauchy-Carlitz numbers $CC_n$ (\cite{KK}) are give by 
\begin{equation} 
\frac{z}{\log_C(z)}=\sum_{n=0}^\infty\frac{CC_n}{\Pi(n)}z^n\,. 
\label{ccarlitz} 
\end{equation}  

The (unsigned) Stirling numbers of the first kind $\stf{n}{k}$ appear in the falling factorial 
$$
x(x-1)\cdots(x-n+1)=\sum_{k=0}^n(-1)^{n-k}\stf{n}{k}x^k
$$ 
and the Stirling numbers of the second kind $\sts{n}{k}$ may be defined by 
\begin{equation} 
x^n=\sum_{k=0}^n x(x-1)\cdots(x-k+1)\sts{n}{k}\,. 
\label{30}
\end{equation}  
Based upon such relations, we can introduce different type Stirling-Carlitz numbers $\stf{n}{k}_A$ and $\sts{n}{k}_A$.  

Throughout this paper, let $\mathbb F_r$ be the field with $r$ elements and 
$\mathbb A=\mathbb F_r[T]$ (resp. $\mathbb F_r(T)$) the ring of  polynomials 
(resp. the field of rational functions) in one variable over $\mathbb F_r$. 
According to the notations used in \cite{Goss}, 
set $[i]:=T^{r^i}-T\in\mathbb A$ ($i\ge 1$), $D_i:=[i][i-1]^r\cdots[1]^{r^{i-1}}$ ($i\ge 1$) with $D_0=1$, and $L_i:=[i][i-1]\cdots[1]$ ($i\ge 1$) with $L_0=1$.  Then, The Carlitz exponential $e_C(x)$ is defined by 
$$  
e_C(x)=\sum_{i=0}^\infty\frac{x^{r^i}}{D_i} 
$$ 
and the Carlitz logarithm $\log_C(x)$ is defined by 
$$  
\log_C(x)=\sum_{i=0}^\infty(-1)^i\frac{x^{r^i}}{L_i}\,.  
$$   
The Carlitz factorial $\Pi(i)$ is defined by 
$$   
\Pi(i)=\prod_{j=0}^m D_j^{c_j}
$$  
for a non-negative integer $i$ with $r$-ary expansion: 
$$   
i=\sum_{j=0}^m c_j r^j\quad (0\le c_j<r)\,. 
$$ 
Denote the $d$-dimensional $\mathbb F_r$-vector space of polynomials of degree$<d$ by $\mathbb A(d):=\{\alpha\in\mathbb A|\deg(\alpha)<d\}$.  
As 
$$
e^{x\log(1+t)}=(1+t)^x=\sum_{n=0}^\infty\binom{x}{n}t^n 
$$ 
with 
$$
\binom{x}{n}=\frac{x(x-1)\cdots(x-n+1)}{n!}\,,
$$ 
as an analogous version,  set  
$$
e_C(z\log_C(x))=\sum_{n=0}^\infty E_n(z)x^{r^n}\,,
$$ 
where 
$$
E_n(z)=\frac{e_n(z)}{D_n}
$$ 
(\cite[Corollary 3.5.3]{Goss}).  In addition, we have 
\begin{equation}
e_n(z)=\prod_{\alpha\in\mathbb A(n)}(z+\alpha)
=\prod_{\alpha\in\mathbb A(n)}(z-\alpha)
=\sum_{i=0}^n\stf{n}{i}_A z^{r^i}\,,   
\label{en}
\end{equation}   
where 
\begin{equation} 
\stf{n}{i}_A=(-1)^{n-i}\frac{D_n}{D_i L_{n-i}^{r^i}} 
\label{ascfirst}
\end{equation}  
(\cite[Theorem 3.1.5]{Goss}). 

As an analogue of the Stirling numbers of the second kind in (\ref{30}),  
it is natural to define $\sts{n}{k}_A$ by 
\begin{equation} 
\sum_{k=0}^n e_k(z)\sts{n}{k}_A=z^{r^n}
\label{31}
\end{equation}     
Then, similarly to the $C$-Stirling-Carlitz numbers $\stf{n}{k}_C$ and $\sts{n}{k}_C$ (\cite[Theorem 5]{KK}),  $A$-Stirling-Carlitz numbers $\stf{n}{k}_A$ and $\sts{n}{k}_A$ satisfy the orthogonal identities. 

\begin{thm}  
For $i\le n$, 
\begin{align}
\sum_{k=i}^n\stf{k}{i}_A\sts{n}{k}_A=\delta_{n,i}\,,
\label{33}\\
\sum_{k=i}^n\sts{k}{i}_A\stf{n}{k}_A=\delta_{n,i}\,.
\label{34} 
\end{align} 
\label{th40} 
\end{thm}  
\begin{proof} 
Since 
\begin{align*} 
z^{r^n}&=\sum_{k=0}^n e_k(z)\sts{n}{k}_A\\
&=\sum_{k=0}^n\sum_{i=0}^k\stf{k}{i}_A z^{r^i}\sts{n}{k}_A\\
&=\sum_{i=0}^n\sum_{k=i}^n\stf{k}{i}_A\sts{n}{k}_A z^{r^i}\,, 
\end{align*} 
by comparing the coefficients of $z^{r^i}$ ($i=0,1,\dots,n$), we get (\ref{33}).  
Since  
\begin{align*}  
e_n(z)&=\sum_{i=0}^n\stf{n}{i}_A z^{r^i}\\
&=\sum_{i=0}^n\stf{n}{i}_A\sum_{k=0}^i e_k(z)\sts{i}{k}_A\\
&=\sum_{k=0}^n\sum_{i=k}^n\stf{n}{i}_A\sts{i}{k}_A e_k(z)\,, 
\end{align*} 
by comparing the coefficients of $e_k(z)$ ($k=n,n-1,\dots,1,0$),  
we have (\ref{34}).  
\end{proof}

$A$-Stirling-Carlitz numbers of the second kind have an explicit expression as those of the first kind in (\ref{ascfirst}).  

\begin{thm}  
For $0\le j\le n$, we have  
$$
\sts{n}{j}_A=\frac{D_n}{D_{j}D_{n-j}^{r^j}}\,. 
$$
\label{th50} 
\end{thm}  
\begin{proof}  
We prove the theorem by induction on $j$. 
If $j=n$, then 
$$
\sts{n}{n}_A=1
$$
by (\ref{33}). Next, we consider the case of $j=n-i$ with $i>0$. 
Using the inductive hypothesis and (\ref{33}), we get 
\begin{align} 
\sts{n}{n-i}_A&=-\sum_{d=0}^{i-1}\stf{n-d}{n-i}_A\sts{n}{n-d}_A \nonumber\\
&=-\sum_{d=0}^{i-1}\frac{(-1)^{i-d}D_{n-d}}{D_{n-i}L_{i-d}^{r^{n-i}}}\frac{D_n}{D_{n-d}D_d^{r^{n-d}}} 
\nonumber \\
&= -\frac{D_n}{D_{n-i}}\sum_{d=0}^{i-1}\frac{(-1)^{i-d}}{L_{i-d}^{r^{n-i}}D_d^{r^{n-d}}}
\label{eqn:a1}
\end{align}
Note for any $l\geq 0$ that 
$$(-1)^{r^l}=-1\,.$$
In fact, if $r$ is even, then $-1=1$ because the characteristic of $\mathbb{F}_r$ is $2$. 
Thus, 
\begin{align*}
-\sum_{d=0}^{i-1}\frac{(-1)^{i-d}}{L_{i-d}^{r^{n-i}}D_d^{r^{n-d}}}
& =
\left(-\sum_{d=0}^{i-1}\frac{(-1)^{i-d}}{L_{i-d}D_d^{r^{i-d}}}\right)^{r^{n-i}}\\
& =
\left(\frac{1}{D_i}-\sum_{d=0}^{i}\frac{(-1)^{i-d}}{L_{i-d}D_d^{r^{i-d}}}\right)^{r^{n-i}}\\
& =
\left(\frac{1}{D_i}-\sum_{a=0}^{i}\frac{(-1)^{a}}{L_{a}D_{i-a}^{r^{a}}}\right)^{r^{n-i}}.
\end{align*}
It is well known for any $l\geq 0$ that 
$$
\sum_{a=0}^l\frac{(-1)^{a}}{L_{a}D_{l-a}^{r^{a}}}=\delta_{l,0}
$$
(for instance, see equation (1.63) in \cite{Koc}). 
Hence we obtain by $i\geq 1$ that 
\begin{align}
-\sum_{d=0}^{i-1}\frac{(-1)^{i-d}}{L_{i-d}^{r^{n-i}}D_d^{r^{n-d}}}
=\frac{1}{D_i^{r^{n-i}}}.
\label{eqn:a2}
\end{align}
Combining (\ref{eqn:a1}) and (\ref{eqn:a2}), we deduce that 
\begin{align*}
\sts{n}{n-i}_A=\frac{D_n}{D_{n-i}D_i^{r^{n-i}}}\,.
\end{align*} 
\end{proof}

\begin{exa}  
By using (\ref{en}), we calculate $\stf{n}{i}_A (i=0,1,\ldots,n)$ in the case of 
$r=3$ and $n=1,2,3$.  By 
$$
e_1(z)=z(z+1)(z-1)=\sum_{i=0}^1\stf{1}{i}_A z^{3^i}\,, 
$$ 
we have 
$$
\stf{1}{0}_A=-1\quad\hbox{and}\quad \stf{1}{1}_A=1\,. 
$$  
If $n=2$, then by 
\begin{align*} 
e_2(z)&=z(z+1)(z-1)(z+T)(z-T)\\
&\quad \times (z+T+1)(z+T-1)(z-T+1)(z-T-1)\\
&=\sum_{i=0}^2\stf{2}{i}_A z^{3^i}\,,
\end{align*} 
we see 
$$
\stf{2}{0}_A=T^6+T^4+T^2,\quad \stf{2}{1}_A=-(T^6+T^4+T^2+1),\quad\hbox{and}\quad \stf{2}{2}_A=1\,. 
$$ 
Moreover, by 
\begin{align*} 
&e_3(z)\\
&=z(z+1)(z-1)(z+T)(z-T)\\
&\quad \times(z+T+1)(z+T-1)(z-T+1)(z-T-1)\\
&\quad\times(z+T^2)(z+T^2+1)(z+T^2-1)(z+T^2+T)(z+T^2-T)\\
&\quad\times (z+T^2+T+1)(z+T^2+T-1)(z+T^2-T+1)(z+T^2-T-1)\\
&\quad\times(z-T^2)(z-T^2+1)(z-T^2-1)(z-T^2+T)(z-T^2-T)\\
&\quad\times (z-T^2+T+1)(z-T^2+T-1)(z-T^2-T+1)(z-T^2-T-1)\\ 
&=\sum_{i=0}^3\stf{3}{i}_A z^{3^i}\,,
\end{align*} 
we obtain 
\begin{align*} 
\stf{3}{0}_A&=-T^{42}-T^{40}-T^{38}-T^{36}+T^{34}+T^{32}+T^{30}+T^{28}\\
&\quad +T^{24}+T^{22}+T^{20}+T^{18}-T^{16}-T^{14}-T^{12}-T^{10}\,\\
\stf{3}{1}_A&=T^{42}+T^{40}+T^{38}-T^{36}-T^{34}-T^{32}\\
&\quad -T^{16}-T^{14}-T^{12}+T^{10}+T^8+T^6\,,\\
\stf{3}{2}_A&=-T^{36}-T^{30}-T^{28}-T^{24}-T^{22}-T^{20}-T^{18}\\
&\quad -T^{16}-T^{14}-T^{12}-T^8-T^6-1\,,\\  
\stf{3}{3}_A&=1\,. 
\end{align*}  
\end{exa}

\section{Applications to hypergeometric Bernoulli and Cauchy numbers}  

The Hasse-Teichm\"uller derivative $H^{(n)}$ of order $n$ is defined by 
$$
H^{(n)}\left(\sum_{m=R}^{\infty} a_m z^m\right)
=\sum_{m=R}^{\infty} a_m \binom{m}{n}z^{m-n}
$$
for $\sum_{m=R}^{\infty} a_m z^m\in \mathbb{F}((z))$, 
where $\mathbb{F}$ is a field of any characteristic, 
$R$ is an integer and $a_m\in\mathbb{F}$ for any $m\geq R$. 

The Hasse-Teichm\"uller derivatives satisfy the product rule \cite{Teich}, the quotient rule \cite{GN} and the chain rule \cite{Hasse}. 
One of the product rules is described as follows:  
\begin{lmm}  
For $f_i\in\mathbb F[[z]]$ ($i=1,\dots,k$) with $k\ge 2$ and for $n\ge 1$, we have 
$$
H^{(n)}(f_1\cdots f_k)=\sum_{i_1,\dots,i_k\ge 0\atop i_1+\cdots+i_k=n}H^{(i_1)}(f_1)\cdots H^{(i_k)}(f_k)\,. 
$$ 
\label{productrule2}
\end{lmm} 

The quotient rules are described as follows:  

\begin{lmm}  
For $f\in\mathbb F[[z]]\backslash \{0\}$ and $n\ge 1$,  
we have 
\begin{align} 
H^{(n)}\left(\frac{1}{f}\right)&=\sum_{k=1}^n\frac{(-1)^k}{f^{k+1}}\sum_{i_1,\dots,i_k\ge 1\atop i_1+\cdots+i_k=n}H^{(i_1)}(f)\cdots H^{(i_k)}(f)
\label{quotientrule1}\\ 
&=\sum_{k=1}^n\binom{n+1}{k+1}\frac{(-1)^k}{f^{k+1}}\sum_{i_1,\dots,i_k\ge 0\atop i_1+\cdots+i_k=n}H^{(i_1)}(f)\cdots H^{(i_k)}(f)\,.
\label{quotientrule2} 
\end{align}   
\label{quotientrules}
\end{lmm}  


In \cite{JKS} Bernoulli numbers and Bernoulli-Carlitz numbers are expressed explicitly by using the Hasse-Teichm\"uller derivative.  In \cite{KK}, Cauchy numbers and Cauchy-Carlitz numbers are expressed explicitly as well.  

In this section, 
by using the Hasse-Teichm\"uller derivative of order $n$, we shall obtain some explicit expressions of the hypergeometric Cauchy numbers $c_{N,n}$, defined by 
$$
\frac{1}{{}_2 F_1(1,N;N+1;-x)}=\sum_{n=0}^\infty c_{N,n}\frac{x^n}{n!}\,, 
$$  
where ${}_2 F_1(a,b;c,z)$ is the hypergeometric function defined by 
$$
{}_2 F_1(a,b;c,z)=\sum_{n=0}^\infty\frac{(a)^{(n)}(b)^{(n)}}{(c)^{(n)}}\frac{z^n}{n!}
$$ 
with $(a)^{(n)}=a(a+1)\cdots(a+n-1)$ ($n\ge 1$) and $(a)^{(0)}=1$.  
We give a different proof for 
the following result shown in \cite[Theorem 1]{Ko3}. 

\begin{thm}  
For $n\ge 1$, 
$$
c_{N,n}=(-1)^n n!\sum_{k=1}^{n}\sum_{i_1,\dots,i_k\ge 1\atop i_1+\cdots+i_k=n}\frac{(-N)^k}{(N+i_1)\cdots(N+i_k)}\,.
$$ 
\label{th_hcauchy1} 
\end{thm}  
\begin{proof}  
Put 
$$
h:={}_2 F_1(1,N;N+1;-x)=N\sum_{j=0}^\infty\frac{(-x)^j}{N+j}\,. 
$$ 
Note that 
\begin{align*} 
\left. H^{(i)}(h)\right|_{x=0}=\left.\sum_{j=0}^\infty\frac{N(-1)^j}{N+j}\binom{j}{i}x^{j-i}\right|_{x=0}
=\frac{N(-1)^i}{N+i}\,. 
\end{align*}    
Hence, by using Lemma \ref{quotientrules} (\ref{quotientrule1}), we have 
\begin{align*} 
\frac{c_{N,n}}{n!}&=\left.H^{(n)}\left(\frac{1}{h}\right)\right|_{x=0}\\
&=\sum_{k=1}^n\left.\frac{(-1)^k}{h^{k+1}}\right|_{x=0}\sum_{i_1,\dots,i_k\ge 1\atop i_1+\cdots+i_k=n}\left.H^{(i_1)}(h)\right|_{x=0}\cdots\left.H^{(i_k)}(h)\right|_{x=0}\\
&=\sum_{k=1}^n(-1)^k\sum_{i_1,\dots,i_k\ge 1\atop i_1+\cdots+i_k=n}\frac{N(-1)^{i_1}}{N+i_1}\cdots\frac{N(-1)^{i_k}}{N+i_k}\\
&=\sum_{k=1}^n\sum_{i_1,\dots,i_k\ge 1\atop i_1+\cdots+i_k=n}\frac{(-N)^k(-1)^n}{(N+i_1)\cdots(N+i_k)}\,. 
\end{align*} 
\end{proof}   

We express the hypergeometric Cauchy numbers in terms of the binomial coefficients, too. 
In fact, by using Lemma \ref{quotientrules} (\ref{quotientrule2}) 
instead of Lemma \ref{quotientrules} (\ref{quotientrule1}) 
in the proof of Theorem \ref{th_hcauchy1}, we obtain the following:  

\begin{prp}  
For $n\ge 1$, 
$$
c_{N,n}=(-1)^n n!\sum_{k=1}^{n}\binom{n+1}{k+1}\sum_{i_1,\dots,i_k\ge 0\atop i_1+\cdots+i_k=n}\frac{(-N)^k}{(N+i_1)\cdots(N+i_k)}\,.
$$ 
\label{th_hcauchy2} 
\end{prp} 

Expressions of $c_{N,n}$ in 
Theorem \ref{th_hcauchy1} and Proposition \ref{th_hcauchy2} are explicit but not convenient to calculate them.  Now, using associated Stirling numbers of the first kind, we introduce a more convenient expression of hypergeometric Cauchy numbers.  
Associated Stirling numbers of the first kind $\stf{n}{k}_{\ge m}$ (\cite{Cha, Ko6, KLM, KMS}) are given by 
\begin{equation} 
\frac{\bigl(-\log(1-x)-F_{m-1}(x)\bigr)^k}{k!}=\sum_{n=0}^\infty\stf{n}{k}_{\ge m}\frac{x^n}{n!}\quad(m\ge 1)\,, 
\label{assoc.stf} 
\end{equation} 
where 
$$
F_m(x)=\begin{cases} 
0&(m=0);\\ 
\sum_{n=1}^m\frac{x^n}{n}&(m\ge 1)\,. 
\end{cases} 
$$     
When $m=1$, $\stf{n}{k}=\stf{n}{k}_{\ge 1}$ is the classical Stirling numbers of the first kind.  
Now, 
we obtain a simple expression for hypergeometric Cauchy numbers in terms of the binomial coefficients and incomplete Stirling numbers of the first kind.

\begin{thm} 
For $N\ge 1$ and $n\ge 1$, we have  
$$
c_{N,n}=(-1)^n n!\sum_{k=1}^n\binom{n+1}{k+1}\frac{(-N)^k k!}{(n+N k)!}\stf{n+N k}{k}_{\ge N}\,. 
$$ 
\label{th26} 
\end{thm}

\noindent 
\begin{rem} 
When $N=1$, Theorem \ref{th26} is reduced to  
$$
c_n=\sum_{k=1}^n\dfrac{(-1)^{n-k}\binom{n+1}{k+1}}{\binom{n+k}{k}}\stf{n+k}{k}\,,
$$
which is Proposition 2 in \cite{KK}.  
\end{rem}

\begin{proof} 
From (\ref{assoc.stf}), we have 
\begin{align*} 
\left(\sum_{j=0}^\infty\frac{t^j}{j+N}\right)^k&=\left(\frac{-\log(1-t)-F_{N-1}(t)}{t^N}\right)^k\\
&=\sum_{n=k}^\infty k!\stf{n}{k}_{\ge N}\frac{t^{n-N k}}{n!}\\
&=\sum_{n=-(N-1)k}^\infty\frac{k!}{(n+N k)!}\stf{n+N k}{k}_{\ge N}t^n\,. 
\end{align*}   
Notice that 
$$
\left.H^{(i)}\left(\frac{-\log(1-t)-F_{N-1}(t)}{t^N}\right)\right|_{t=0}=\frac{1}{i+N}\,. 
$$ 
Applying Lemma \ref{productrule2} with 
$$
f_1(t)=\cdots=f_k(t)=\frac{-\log(1-t)-F_{N-1}(t)}{t^N}\,, 
$$
we get 
\begin{equation}  
\frac{k!}{(n+N k)!}\stf{n+N k}{k}_{\ge N}=\sum_{i_1,\dots,i_k\ge 0\atop i_1+\cdots+i_k=n}\frac{1}{(i_1+N)\cdots(i_k+N)}\,. 
\label{eq20} 
\end{equation}   
Together with Proposition \ref{th_hcauchy2}, we get the desired result.  
\end{proof}  

\begin{exa}  
Let $N=3$ and $n=4$. 
By the definition in (\ref{assoc.stf}), we get 
\begin{align*} 
&\frac{1}{7!}\stf{7}{1}_{\ge 3}=\frac{1}{7},\quad \frac{1}{10!}\stf{10}{2}_{\ge 3}=\frac{153}{1400},\quad 
\frac{1}{13!}\stf{13}{3}_{\ge 3}=\frac{1751}{50400}\\ 
&\hbox{and}\quad \frac{1}{16!}\stf{16}{4}_{\ge 3}=\frac{190261}{29030400}\,.
\end{align*}  
Hence, 
\begin{align*} 
c_{3,4}&=4!\sum_{k=1}^4\binom{5}{k+1}\frac{(-3)^k k!}{(4+3 k)!}\stf{4+3 k}{k}_{\ge 3}\\
&=4!\left(-\binom{5}{2}\frac{3}{7!}\stf{7}{1}_{\ge 3}+\binom{5}{3}\frac{3^2\cdot 2}{10!}\stf{10}{2}_{\ge 3}\right.\\
&\quad\left.-\binom{5}{4}\frac{3^3\cdot 3!}{13!}\stf{13}{3}_{\ge 3}+\binom{5}{5}\frac{3^4\cdot 4!}{16!}\stf{16}{4}_{\ge 3}\right)\\
&=-\frac{1971}{5600}\,. 
\end{align*} 
\end{exa}


Next, we shall obtain some explicit expressions of the hypergeometric Bernoulli numbers $B_{N,n}$ (\cite{HN1,HN2,Kamano}), defined by 
$$
\frac{1}{{}_1 F_1(1;N+1;x)}=\sum_{n=0}^\infty B_{N,n}\frac{x^n}{n!}\,, 
$$  
where ${}_1 F_1(a;b;z)$ is the confluent hypergeometric function defined by 
$$
{}_1 F_1(a;b;z)=\sum_{n=0}^\infty\frac{(a)^{(n)}}{(b)^{(n)}}\frac{z^n}{n!}\,.
$$ 
Then we have the following: 

\begin{thm}  
For $n\ge 1$, 
$$
B_{N,n}=n!\sum_{k=1}^{n}(-N!)^k\sum_{i_1,\dots,i_k\ge 1\atop i_1+\cdots+i_k=n}\frac{1}{(N+i_1)!\cdots(N+i_k)!}\,.
$$ 
\label{th_hbernoulli1} 
\end{thm}  
\begin{proof}  
Put 
$$
h:={}_1 F_1(1;N+1;x)=\sum_{j=0}^\infty\frac{N! x^j}{(N+j)!}\,. 
$$ 
Note that 
\begin{align*} 
\left. H^{(i)}(h)\right|_{x=0}=\left.\sum_{j=0}^\infty\frac{N!}{(N+j)!}\binom{j}{i}x^{j-i}\right|_{x=0}
=\frac{N!}{(N+i)!}\,. 
\end{align*}    
Hence, by using Lemma \ref{quotientrules} (\ref{quotientrule1}), we have 
\begin{align*} 
\frac{B_{N,n}}{n!}&=\left.H^{(n)}\left(\frac{1}{h}\right)\right|_{x=0}\\
&=\sum_{k=1}^n\left.\frac{(-1)^k}{h^{k+1}}\right|_{x=0}\sum_{i_1,\dots,i_k\ge 1\atop i_1+\cdots+i_k=n}\left.H^{(i_1)}(h)\right|_{x=0}\cdots\left.H^{(i_k)}(h)\right|_{x=0}\\
&=\sum_{k=1}^n(-1)^k\sum_{i_1,\dots,i_k\ge 1\atop i_1+\cdots+i_k=n}\frac{N!}{(N+i_1)!}\cdots\frac{N!}{(N+i_k)!}\\
&=\sum_{k=1}^n\sum_{i_1,\dots,i_k\ge 1\atop i_1+\cdots+i_k=n}\frac{(-N!)^k}{(N+i_1)!\cdots(N+i_k)!}\,. 
\end{align*} 
\end{proof}   

We express the hypergeometric Bernoulli numbers in terms of the binomial coefficients, too. 
In fact, by using Lemma \ref{quotientrules} (\ref{quotientrule2}) 
instead of Lemma \ref{quotientrules} (\ref{quotientrule1}) 
in the proof of Theorem \ref{th_hbernoulli1}, we obtain the following:  

\begin{prp}  
For $n\ge 1$, 
$$
B_{N,n}=n!\sum_{k=1}^{n}(-N!)^k\binom{n+1}{k+1}\sum_{i_1,\dots,i_k\ge 0\atop i_1+\cdots+i_k=n}\frac{1}{(N+i_1)!\cdots(N+i_k)!}\,.
$$ 
\label{th_hbernoulli2} 
\end{prp} 

In the same way as the proof of Theorem \ref{th26}, 
using associated Stirling numbers of the second kind, we introduce a more convenient expression of hypergeometric Cauchy numbers. 
Associated Stirling numbers of the second kind $\sts{n}{k}_{\ge m}$ (\cite{Cha, Ko6, KLM, KMS}) are given by 
\begin{equation} 
\frac{\bigl(e^x-E_{m-1}(x)\bigr)^k}{k!}=\sum_{n=0}^\infty\sts{n}{k}_{\ge m}\frac{x^n}{n!}\quad(m\ge 1)\,, 
\label{assoc.sts} 
\end{equation} 
where 
$$
E_m(x)=\sum_{n=0}^m\frac{x^n}{n!}\,. 
$$     
When $m=1$, $\sts{n}{k}=\sts{n}{k}_{\ge 1}$ is the classical Stirling numbers of the second kind.  
Now, 
we obtain a simple expression for hypergeometric Bernoulli numbers in terms of the binomial coefficients and incomplete Stirling numbers of the second kind.

\begin{thm} 
For $N\ge 1$ and $n\ge 1$, we have  
$$
B_{N,n}=n!\sum_{k=1}^n\binom{n+1}{k+1}\frac{(-N!)^k k!}{(n+N k)!}\sts{n+N k}{k}_{\ge N}\,. 
$$ 
\label{th27} 
\end{thm}

\begin{rem} 
When $N=1$, Theorem \ref{th27} is reduced to  
$$
B_n=\sum_{k=1}^n\dfrac{(-1)^{k}\binom{n+1}{k+1}}{\binom{n+k}{k}}\sts{n+k}{k}\,,
$$
which is a simple formula of Bernoulli numbers, appeared in \cite{Gould,SS}. 
\end{rem}

\begin{proof} 
From (\ref{assoc.sts}), we have 
\begin{align*} 
\left(\sum_{j=0}^\infty\frac{t^j}{(j+N)!}\right)^k&=\left(\frac{e^t-E_{N-1}(t)}{t^N}\right)^k\\
&=\sum_{n=k}^\infty k!\sts{n}{k}_{\ge N}\frac{t^{n-N k}}{n!}\\
&=\sum_{n=-(N-1)k}^\infty\frac{k!}{(n+N k)!}\sts{n+N k}{k}_{\ge N}t^n\,. 
\end{align*}   
Notice that 
$$
\left.H^{(i)}\left(\frac{e^t-E_{N-1}(t)}{t^N}\right)\right|_{t=0}=\frac{1}{(i+N)!}\,. 
$$ 
Applying Lemma \ref{productrule2} with 
$$
f_1(t)=\cdots=f_k(t)=\frac{e^t-E_{N-1}(t)}{t^N}\,, 
$$
we get 
\begin{equation}  
\frac{k!}{(n+N k)!}\sts{n+N k}{k}_{\ge N}=\sum_{i_1,\dots,i_k\ge 0\atop i_1+\cdots+i_k=n}\frac{1}{(i_1+N)!\cdots(i_k+N)!}\,. 
\label{eq201} 
\end{equation}   
Together with Proposition \ref{th_hbernoulli2}, we get the desired result.  
\end{proof}  


\end{document}